\documentclass{article}

\usepackage{maa-monthly}

\usepackage{tikz}
\usetikzlibrary{calc,shapes.geometric} % need to load the package that lets me use ellipses
\usepackage{mathdots} % need to load a package to draw BL-TR dots
\usepackage{tikz-cd}
\usepackage{dsfont}
\usepackage{epstopdf}
\usepackage{enumitem}
\usepackage{caption}
\usepackage{subcaption}

\tikzset{
base/.style = {circle, draw=black, text centered},
>=Latex
}

\newtheorem{theorem}{Theorem}
\newtheorem{proposition}[theorem]{Proposition} % my own
\newtheorem{lemma}[theorem]{Lemma} % my own
\newtheorem{corollary}[theorem]{Corollary} % my own
 
\theoremstyle{definition}
\newtheorem*{definition}{Definition}
\newtheorem*{remark}{Remark}

\newcommand{\Q}{\mathbb{Q}}
\newcommand{\C}{\mathbb{C}}

\newcommand{\Z}{\mathbb{Z}}

\newcommand{\cJ}{\mathcal{J}}

\newcommand{\abs}[1]{\left\lvert #1 \right\rvert}

\newcommand{\set}[2]{\left\{#1 \ : \ #2\right\}}

\newcommand{\conv}[1]{\underset{#1}\longrightarrow}
\DeclareMathOperator{\id}{id}
\DeclareMathOperator{\subspan}{span}

\DeclareMathOperator{\Imag}{Im}

\newcommand\restr[2]{{% we make the whole thing an ordinary symbol
  \left.\kern-\nulldelimiterspace % automatically resize the bar with \right
  #1 % the function
  \vphantom{\big|} % pretend it's a little taller at normal size
  \right|_{#2} % this is the delimiter
  }}

\begin{document}

\title{Dynamical spectrum via determinant-free linear algebra}
\author{Joseph Horan}
\date{\today}

\maketitle

\begin{abstract}
We consider a sequence of matrices that are associated to Markov dynamical systems and use determinant-free linear algebra techniques (as well as some algebra and complex analysis) to rigorously estimate the eigenvalues of every matrix simultaneously without doing any calculations on the matrices themselves. As a corollary, we obtain mixing rates for every system at once, as well as symmetry properties of densities associated to the system; we also find the spectral properties of a sequence of related factor systems.
\end{abstract}

% What's the outline?

% Intro, like Anthony said a paragraph on the place where the dynamics arise and why the spectrum matters, a quick note on how the spectrum reduces (via Markov), just getting to the collection of matrices (alongside the general form of the matrix and a couple Markov partition pictures like in the other paper) as quickly as possible while still be interesting and plausible

% first goal is to study the spectrum of the adjacency matrix, so in order I have to: find the characteristic polynomial, factor it, analyse each factor to see that there's the largest/second-largest e-vals and two at zero and the rest near the unit circle (using the simple inequalities and Rouch\'e's theorem), supplement with the cool pictures of the spectrum. also supplement with, for some small n, what the eigenvectors are? to see how some are obv and some really aren't. (pg. 19 in my written notes for another outline)

% second goal is to study the J/flip map/symmetry of the system, so I have to introduce J, identify spectral properties in relation to the commutation relation between J and A_n, see how sym e-vecs are with e-vals of f_n and anti-sym e-vecs are with e-vals of g_n, side-comment about J having a nice description using a tensor product, look at the minimal polynomial and how the C[x]-action of A_n on V_n is the same as the C[x,y]-action of (A_n,J_n) on V_n, which is the same as the C[x,y]/<y^2-1>-action of (A_n,J_n) on V_n, look at the factor map... maybe not in that order.

\section*{Introduction}

Consider, for the time being, a \emph{stochastic} $d$-by-$d$ matrix $P$. The matrix $P$ represents a finite-dimensional \emph{Markov chain}, a stochastic model where states transition to one another with some probability at discrete time steps according to the entries in the matrix. Thus, if at time $0$ the probabilities of being in each of the $d$ states are given by the vector $x$, then the probabilities of being in each of the $d$ states at time $1$ are given by $Px$ ($P$ acting on $x$); see Figure \ref{fig:markov-chain}. The asymptotic properties of the Markov chain, such as what the stationary distribution is (if it exists) and the rate at which the process converges to that distribution, are determined by the spectral theory of the matrix $P$. Some linear algebra, potentially including some numerical computation, then allows us to compute these desired quantities. In particular, in the case that the Markov chain is \emph{mixing}, we wish to find the modulus of the second-largest eigenvalue(s), which tells us the rate at which the Markov chain converges to its stationary distribution: the mixing time is at most proportional to the reciprocal of the logarithm of the modulus of the second-largest eigenvalue.\footnote{For the proof of this fact and for more on Markov chains, see the book by Levin, Peres, and Wilmer \cite{mc-mixing}; applications include statistical mechanics and Markov chain Monte Carlo (MCMC).}

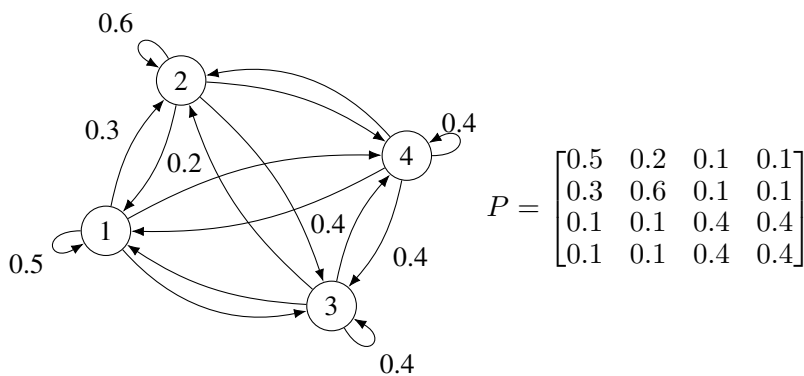
\begin{figure}[bp]
\centering
\begin{subfigure}{0.45\textwidth}
\centering
\begin{tikzpicture}[scale=1]
\node (state1) at (0,0) [base] {1};
\node (state2) at (1,2) [base] {2};
\node (state3) at (3,-1) [base] {3};
\node (state4) at (4,1) [base] {4};

\draw[->] (state1) to [bend left=15] node[auto] {0.3} (state2);
\draw[->] (state2) to [bend left=15] node[right] {0.2} (state1);
\draw[->] (state1) to [out=180,in=210,loop] node[auto,swap] {0.5} (state1);
\draw[->] (state2) to [out=120,in=150,loop] node[auto,swap] {0.6} (state2);

\draw[->] (state1) to [bend right=30] node[auto] {} (state3);
\draw[->] (state3) to [bend left=15] node[auto] {} (state1);
\draw[->] (state2) to [bend left=15] node[auto] {} (state4);
\draw[->] (state4) to [bend right=30] node[auto] {} (state2);

\draw[->] (state2) to [bend left=15] node[auto] {} (state3);
\draw[->] (state1) to [bend left=15] node[auto] {} (state4);
\draw[->] (state4) to [bend left=15] node[auto] {} (state1);
\draw[->] (state3) to [bend left=15] node[auto] {} (state2);

\draw[->] (state3) to [bend left=15] node[left] {0.4} (state4);
\draw[->] (state4) to [bend left=15] node[auto] {0.4} (state3);
\draw[->] (state3) to [out=300,in=330,loop] node[auto,swap] {0.4} (state3);
\draw[->] (state4) to [out=0,in=30,loop] node[above] {0.4} (state4);
\end{tikzpicture}
\end{subfigure}
\begin{subfigure}{0.45\textwidth}
\centering
\[P = \begin{bmatrix}
0.5 & 0.2 & 0.1 & 0.1 \\
0.3 & 0.6 & 0.1 & 0.1 \\
0.1 & 0.1 & 0.4 & 0.4 \\
0.1 & 0.1 & 0.4 & 0.4
\end{bmatrix} \]
\end{subfigure}
\caption{A Markov chain with four states and its associated transition matrix.}
\label{fig:markov-chain}
\end{figure}

In the field of dynamical systems, we often start with a map $T$ on some state space $X$, and we want to answer questions such as ``what happens to most of the orbits of $T$ over a long time?'' and ``do regions of $X$ mix together over time, and at what rate?'' These questions are less about looking at individual orbits of points under $T$ and more about looking at what happens \emph{on average}. Specifically, we can learn much about the dynamical system $(X,T)$ by studying how \emph{probability densities} on $X$ change over time under the action of $T$. 

To formalize this process and to lead into the focus of this article, we consider a specific class of piecewise linear maps acting on $[-1,1]$.
\begin{definition} 
Define $T_{\kappa} : [-1,1] \to [-1,1]$ by:
\[ T_{\kappa}(x) = \begin{cases}
2(1+\kappa)(x+1) - 1, & x\in [-1,-1/2], \\
-2(1+\kappa)x - 1, & x\in [-1/2,0), \\
0, & x = 0, \\
-2(1+\kappa)x + 1, & x\in (0,1/2], \\
2(1+\kappa)(x-1) + 1, & x\in [1/2,1].
\end{cases} \]
We call $T_{\kappa}$ a \emph{paired tent map}, because there are two tents paired together. See Figure \ref{fig:ptm-eg1} for an illustration.\footnote{In general, the tents could be different; see Section 4 of \cite{horan-pf-thm}.}
\end{definition}

\begin{figure}[tbp]
\centering
\includegraphics[width=0.75\textwidth]{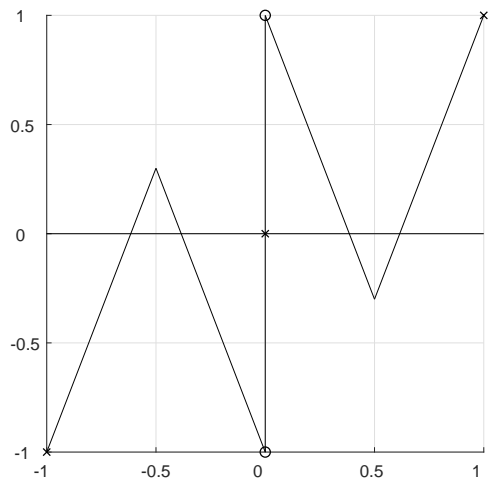}
\captionsetup{width=0.75\textwidth,labelfont=bf}
\caption{The paired tent map, with parameter $\kappa = 0.3$.}
\label{fig:ptm-eg1}
\end{figure}

Let $f$ be a probability density on $[-1,1]$; that is, a non-negative measurable function defined on $[-1,1]$ with integral equal to $1$. As a rough analogy, one could imagine that the space $[-1,1]$ is a bowl of banana bread batter into which one has placed chocolate chips, and $f$ is the density of chocolate chips. Applying the map $T_{\kappa}$ stirs the space up, moving the chocolate chips around; there is then a new density, call it $P_{\kappa}f$, that describes the new locations of the chocolate chips. Some parts of the batter may have more chocolate chips than before, and some fewer, but the total amount of chocolate chips has not changed. It turns out that the operator $P_{\kappa}$ can be defined on all integrable functions (that is, on $L^1(\lambda)$, where $\lambda$ is the normalized Lebesgue measure), and is bounded and linear; we call $P_{\kappa}$ the \emph{Perron-Frobenius operator associated to} $T_{\kappa}$.\footnote{To be rigorous, $P_{\kappa}f$ is the Radon-Nikodym derivative of the measure $A \mapsto \lambda(f\mathds{1}_{T^{-1}(A)})$, which exists because $T_{\kappa}$ is a non-singular map. See, for example, Chapter 4 of \cite{gora-boyarsky}.} It also turns out that there is an invariant subspace for $P_{\kappa}$ of $L^1$ called $BV$ (short for \emph{bounded variation}) on which the spectrum of $P_\kappa$ is well-behaved, and so we restrict our focus to $BV$ for the remainder of this article.\footnote{As shown by Ding, Du, and Li \cite{li-spectrum}, Perron-Frobenius operators can have $L^1$-spectrum equal to the entire closed unit disk; the $BV$-spectrum is significantly more reasonable.} % need a citation for the technical details, eg G+B, but it'd also be good to have a sample picture indicating how this whole thing happens in one step. reasonable to mention optimal transport/mass transport, here?

Returning to the questions posed above, we note that $P_{\kappa}$ is the infinite-dimensional analogue of the transition matrix $P$ for the Markov chain. If we want to find a ``stationary distribution'' for $T_{\kappa}$, we really are looking for invariant densities, which are eigenvectors of $P_{\kappa}$ with eigenvalue $1$. If all initial densities converge to an invariant density over time, then we have a good idea of where most of the points in $[-1,1]$ end up in the long run: no matter where they started, points will be distributed over $[-1,1]$ according to the invariant density. Moreover, if there is a gap in modulus between an eigenvalue of $1$ and the rest of the spectrum, this gap describes how quickly this convergence occurs, in the same way as described above for Markov chains.

By inspection, if $\kappa = 0$, then from the graph of $T_{\kappa}$ it is clear that $T_{\kappa}$ has \emph{two} invariant densities: the characteristic functions on $[-1,0]$ and $[0,1]$, respectively. If $\kappa > 0$, we can see that these two densities are no longer invariant, because there is mixing between the two intervals $[-1,0]$ and $[0,1]$. It is \emph{a priori} unclear whether or not $T_{\kappa}$ has an invariant density, and if it does whether it has a spectral gap; however, to answer these questions we can study $P_{\kappa}$, as described above.

\section*{Markov Maps and Partitions}

Unfortunately, the fact that $P_{\kappa}$ is not a matrix complicates things; at first glance, we no longer have all of the computational and theoretical tools available to us previously. However, because $T_{\kappa}$ is piecewise-linear, if the map $T_{\kappa}$ has an additional property then we can recover a significant portion of our toolkit.
\begin{definition}
\label{defn:markov}
The map $T_{\kappa}$ is \emph{Markov} when there is a finite collection $\{R_i\}_{i=1}^r$ of disjoint open intervals in $[-1,1]$ such that:
\begin{enumerate}
\item $[-1,1] \setminus \bigcup_i R_i$ is the collection of endpoints of the intervals $\{R_i\}$, and
\item if $R_i$ intersects $T_{\kappa}(R_j)$, then all of $R_i$ is contained in $T_{\kappa}(R_j)$.
\end{enumerate} 
The collection $\{R_i\}$ is called a \emph{Markov partition} for $T_{\kappa}$, even though it is not a partition, strictly speaking.
\end{definition}

The next lemma is a combination of Theorem 9.2.1 in \cite{gora-boyarsky} and Lemma 3.1 in \cite{blank-keller}, stated in the specific case of our paired tent maps $T_{\kappa}$.

\begin{lemma}
\label{lem:markov-op}
Suppose that the paired tent map $T_{\kappa}$ is Markov, with Markov partition $\{R_i\}_{i=1}^r$. If $V = \subspan_{\C}\set{\mathds{1}_{R_i}}{1 \leq i \leq r}$ and $P_{\kappa}$ is the Perron-Frobenius operator for $T_{\kappa}$, then $V$ is $P_{\kappa}$-invariant (considered as a subspace of $BV$). The adjacency matrix for $T_{\kappa}$ is given by the $r$-by-$r$ matrix $A_{\kappa} = \left[ a_{ij} \right]$, where \[ a_{ij} = \begin{cases}
1, & R_i^o \subset T(R_j), \\
0, & \text{otherwise}.
\end{cases} \] Define an isomorphism $\phi_r : V \to \C^r$ by $\phi(\mathds{1}_{R_i}) = e_i$. Then the restriction of $P_{\kappa}$ to $V$ can be represented by the $r$-by-$r$ matrix $M_{\kappa} = (2(1+\kappa))^{-1}A_{\kappa}$, with $\phi_r\circ P_{\kappa} = M_{\kappa}\circ \phi_r$. Moreover, \[ \sigma(P_{\kappa})\setminus \overline{B(0,(2(1+\kappa))^{-1})} = %\sigma\left( \restr{P_{\kappa}}{V} \right)\setminus \overline{B(0,(2(1+\kappa))^{-1})} = \sigma(M_{\kappa})\setminus \overline{B(0,(2(1+\kappa))^{-1})} = 
(2(1+\kappa))^{-1}\sigma(A_{\kappa})\setminus \overline{B(0,(2(1+\kappa))^{-1})}. \] %Finally, the spectral radius of $M_{\kappa}$ is equal to $1$, so the spectral radius of $A_{\kappa}$ is $2(1+\kappa)$.
\end{lemma}

% soo... the matrix thing is from G+B, chapter whatever. The spectrum thing for LE's is found in FLQ 1 section 3, referencing Lemma 3.1 in Blank and Keller from 1998 for the autonomous case. The essential spectral radius is, in this cases, (2(1+\kappa))^{-1}, because of a Keller 1984 thing which references Rychlik. Write the proof for the thesis, probably.

We see that when $T_{\kappa}$ is Markov, to find the largest eigenvalues for $P_{\kappa}$ it suffices to look only at the spectrum of the matrix $A_{\kappa}$, for which we have all of our linear algebra tools. In particular, we can look at the spectrum of $A_{\kappa}$ to find the second-largest eigenvalues. So, we ask: when are these maps Markov? A general sufficient condition for piecewise linear maps is given by the following lemma, which says that it is enough for the endpoints of monotonicity intervals to be invariant in finitely many steps. We may then apply the lemma to $T_{\kappa}$ by investigating the images of $\pm 1/2$. Recall that $T(x^+)$ is the limit $\lim_{y\to x^+} T(y)$, and similarly for $T(x^-)$.

\begin{lemma}
\label{lem:markov-points}
Let $T : [a,b] \to [a,b]$ be an onto piecewise linear map, and let $E_0$ be the set of endpoints of the intervals of monotonicity for $T$. For each $i\geq 1$, let $E_i = \set{T(s^{\pm})}{s\in E_{i-1}}$. Suppose that there exists $m$ such that $E_m = E_{m+1}.$ Then $T$ is Markov, with Markov partition $\{R_i\}_{i=1}^R$, where $\{r_i\}_{i=0}^M$ enumerates $E_m$ in an increasing way and $R_i = (r_{i-1},r_i)$.
\end{lemma}

\begin{proof}
Let $m$ be the smallest $m$ such that $E_m = E_{m+1}$; let $r_i$ and $R_i$ be defined as in the statement of the lemma. Since the union of the intervals $\{R_i\}$ and their endpoints is the same as the union of the intervals of monotonicity along with those endpoints, $[-1,1]\setminus \bigcup_i R_i$ is the endpoints of the $R_i$. Then, since $E_m = E_{m+1}$, for each $j$ we have $T(R_j) = (r_k,r_l)$ for some $k < l$ depending on $j$. Thus, if $R_i \cap T(R_j) \ne \emptyset$, we must have $k\leq i < l$, since the intervals $R_i$ are disjoint; hence $R_i \subset T(R_j)$. Hence $T$ is Markov. 
\end{proof}

\begin{lemma}
\label{lem:paired-tent-markov}
There exists a decreasing sequence $(\kappa_n)_{n=1}^{\infty} \subset (0,1/2)$ such that $T_{\kappa_n}$ is Markov and $\kappa_n \conv{n\to\infty} 0$. Each $\kappa_n$ satisfies $(2+2\kappa)^{n}\kappa = 1$. The Markov partition for $T_{\kappa_n}$ is, for $n=1$, \[ \Big\{ \left( -1,-\tfrac{1}{2} \right), \left( -\tfrac{1}{2}, -\kappa_1 \right), \left( -\kappa_1, 0 \right), \left( 0, \kappa_1 \right), \left( \kappa_1, \tfrac{1}{2} \right), \left( \tfrac{1}{2}, 1 \right) \Big\} \] and for $n \geq 2$,
\begin{gather*}
\Big\{ \left( -1, T_{\kappa_n}(-\kappa_n) \right) \Big\} \cup \Big\{ \left( T_{\kappa_n}^{i}(-\kappa_n), T_{\kappa_n}^{i+1}(-\kappa_n) \right) \Big\}_{i=1}^{n-2}  \\
\cup\ \Big\{ \left( T_{\kappa_n}^{n-1}(-\kappa_n),-\tfrac{1}{2} \right), \left( -\tfrac{1}{2}, -\kappa_n \right), \left( -\kappa_n, 0 \right), \left( 0, \kappa_n \right), \left( \kappa_n, \tfrac{1}{2} \right), \left( \tfrac{1}{2}, T_{\kappa_n}^{n-1}(\kappa_n) \right) \Big\} \\
\cup\ \Big\{ \left( T_{\kappa_n}^{i+1}(\kappa_n), T_{\kappa_n}^i(\kappa) \right) \Big\}_{i=1}^{n-2} \cup \Big\{ \left( T_{\kappa_n}(\kappa_n), 1 \right) \Big\}.
\end{gather*}
\end{lemma}

\begin{remark}
\label{rem:markov-part-fig}
The Markov partitions for $T_1$ and $T_4$, are shown in Figures \ref{fig:markov-part-1} and \ref{fig:markov-part-4}. The case $n=1$ is distinct because the branch of the map used for $x=\kappa_n$ is different than the branches used for the further iterates $T_{\kappa_n}^i(\kappa_n)$. In each picture, one may visually confirm that the collection of intervals actually is a Markov partition by checking that the image of each (horizontal) interval stretches vertically over a union of consecutive intervals (at each endpoint, the graph of the map passes through intersection points of horizontal and vertical lines).
\end{remark}

\begin{figure}[tbp]
\centering
\begin{subfigure}{0.475\textwidth}
\centering
\includegraphics[width=\textwidth]{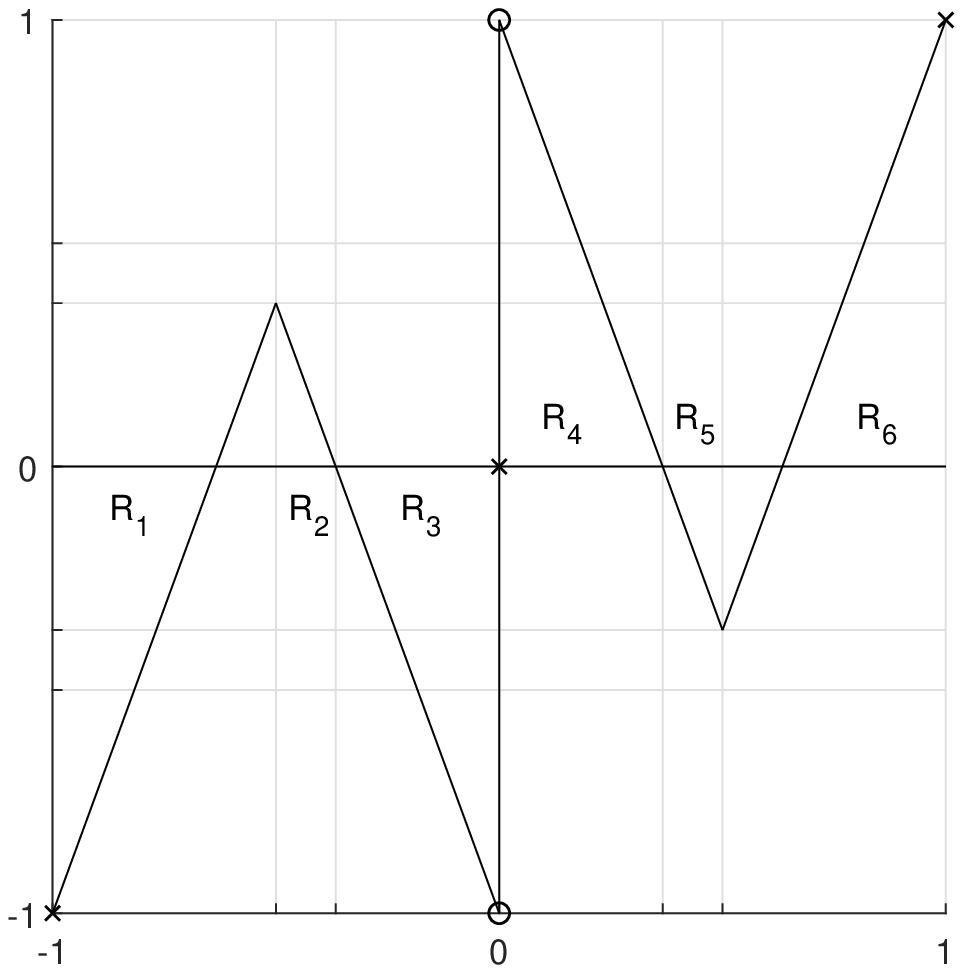}
\caption{$n=1$}
\label{fig:markov-part-1}
\end{subfigure}
\hfill
\begin{subfigure}{0.475\textwidth}
\centering
\includegraphics[width=\textwidth]{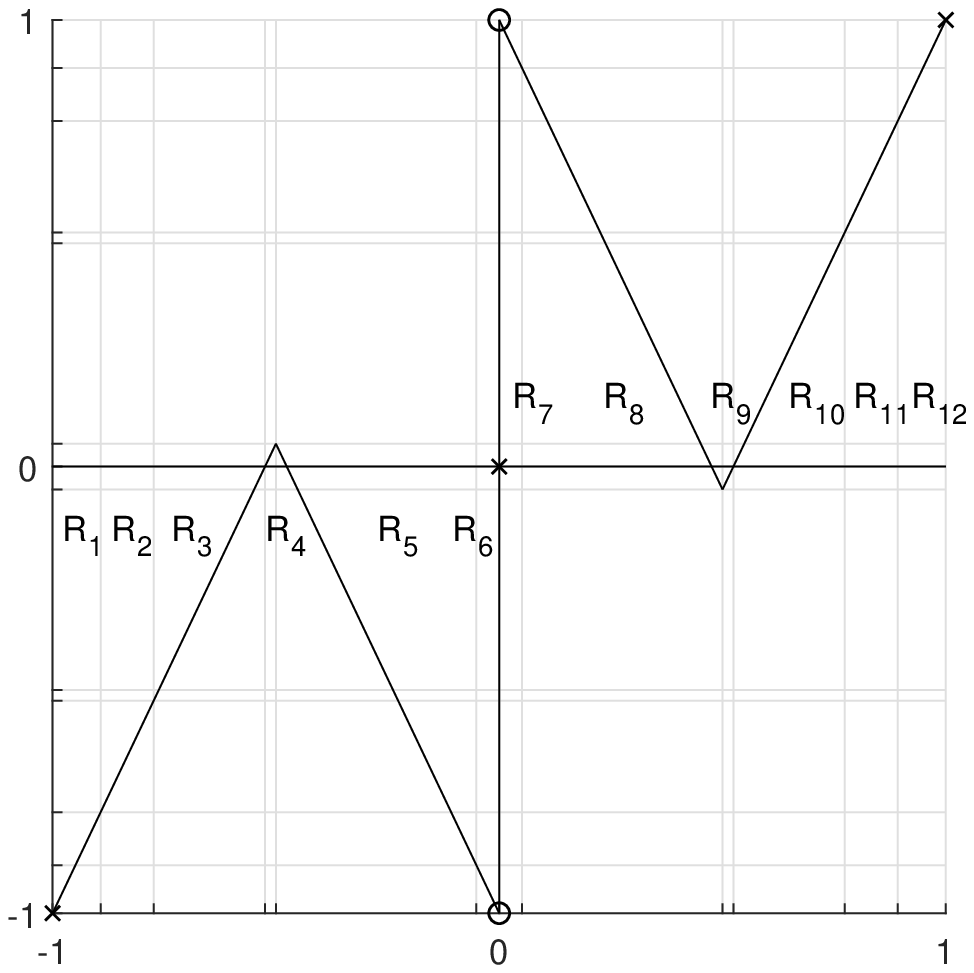}
\caption{$n=4$}
\label{fig:markov-part-4}
\end{subfigure}
%\captionsetup{width=0.75\textwidth,labelfont=bf}
\caption{Markov partitions for $T_{\kappa_n}$, with $n=1,4$.}
\label{fig:markov-part}
\end{figure}

\begin{proof}
Consider the paired tent map $T_{\kappa}$ for $\kappa \in (0,1/2)$. We will use Lemma \ref{lem:markov-points} to find conditions on $\kappa$ that make $T_{\kappa}$ Markov. Start with $E_0 = \{-1,-1/2,0,1/2,1\}.$ The map $T_{\kappa}$ is continuous everywhere except at $0$, for which the one-sided limits are $\pm 1$, and we have $T_{\kappa}(-1) = -1$ and $T_{\kappa}(1) = 1$. Then $T_{\kappa}(-1/2) = \kappa$ and $T_{\kappa}(1/2) = -\kappa$, so we consider iterates of $\kappa$ under $T_{\kappa}$; by symmetry, iterates of $-\kappa$ will work similarly. In particular, we will find, for each $n\geq 1$, a $\kappa_n$ such that $T_{\kappa_n}^{i}(\kappa_n) = 1 - (2+2\kappa_n)^{i}\kappa_n > 1/2$ for $1\leq i<n$ and \[ T_{\kappa_n}^n(\kappa_n) = 1 - (2+2\kappa)^n\kappa = 0. \]

First, consider the equation $(2+2\kappa)^n\kappa - 1 = 0$. Rearrange and take logarithms to obtain \[ n = h(\kappa) := \frac{-\log(\kappa)}{\log(2+2\kappa)}.\] The function $h(\kappa)$ is decreasing on $(0,\infty)$, is unbounded as $\kappa$ tends to $0$, and has $h(1/2) = \log(2)/\log(3) < 1$. Thus we conclude that for each $n\in \Z_{\geq 1}$, there exists a unique $\kappa_n \in (0,1/2)$ solving $(2+2\kappa)^n\kappa - 1 = 0$, and $\kappa_n$ decreases to $0$.\footnote{Another way to see this claim is by observing that $(n,\kappa) \mapsto T_{\kappa}^n(1/2)$ is increasing, in both $n$ and $\kappa$ (where this property makes sense).}

Fix $n \geq 2$; we will show that $T_{\kappa_n}^i(\kappa_n) > 1/2$ for each $i = 1,\dots,n-1$. For each of those $i$, we have: \[ (2+2\kappa_n)^{i}\kappa_n = \frac{1}{(2+2\kappa_n)^{n-i}} \leq \frac{1}{2+2\kappa} < \frac{1}{2}. \] Observe that for $i=1$, we have \[ T_{\kappa_n}(\kappa_n) = 1 - (2+2\kappa_n)\kappa_n > \frac{1}{2}. \] By repeated application of $T_{\kappa_n}$ and use of the upper bound on $(2+2\kappa_n)^{i}\kappa_n$, we see that for all $1 \leq n-1$, \[ T_{\kappa_n}^{i}(\kappa_n) = 1 - (2+2\kappa_n)^{i}\kappa_n > \frac{1}{2}. \] We have proven that $T_{\kappa_n}^n(\kappa_n) = 0$, and $T_{\kappa_n}^i(\kappa_n) > 1/2$ for $1 \leq i < n$; the symmetric statement holds for $-\kappa_n$. 

Finally, fix $n\geq 1$. We claim that $E_n$ is where the sequence of $E_i$ terminates. To see this, observe that \[ E_n = E_0 \cup \{T_{\kappa_n}^{i}(\pm 1/2)\}_{i=1}^{n} = \{-1,-1/2,0,1/2,1\} \cup \{T_{\kappa_n}^{i}(\pm \kappa_n)\}_{i=0}^{n-1}, \] and note that we just saw that $T_{\kappa_n}^n(\pm \kappa_n) = 0$, so that $E_{n+1} = E_n$. This shows that $T_{\kappa_n}$ is Markov, using Lemma \ref{lem:markov-points}. The listed Markov partitions are given by tracing $T_{\kappa_n}^i(\pm \kappa_n)$ as $i$ runs from $0$ to $n-1$.
\end{proof}

We now see that $T_n := T_{\kappa_n}$ is Markov for each $n \geq 1$. From the graph of the maps and the form of the Markov partition, it is easy to read off the adjacency matrix $A_n := A_{\kappa_n}$; since the partition has $2n+4$ pieces, the matrix is $(2n+4)$-by-$(2n+4)$. For $n\geq 4$ the general form of the matrix is as in Figure \ref{fig:adj-matrix}. For $n\leq 3$ some of the columns are combined.

\begin{figure}[tbp]
\[ \left[\ \begin{array}{@{}c|c@{}}
\begin{matrix}
1 & 0 & 0 & & 0 & 0 & 0 & 1  \\
1 & 0 & 0 & & 0 & 0 & 1 & 0 \\
0 & 1 & 0 & & 0 & 0 & 1 & 0 \\
0 & 0 & 1 & & 0 & 0 & 1 & 0 \\
& \vdots & & \ddots & & & \vdots & \\
0 & 0 & 0 & & 1 & 0 & 1 & 0 \\
0 & 0 & 0 & & 1 & 0 & 1 & 0 \\
0 & 0 & 0 & & 1 & 0 & 1 & 0
\end{matrix}
&
\begin{matrix}
0 & 0 & 0 & 0 & & 0 & 0 & 0 \\
& & & & & & & \\
& & & & & & & \\
& & & & & & & \\
& & & & \vdots & & & \\
& & & & \hphantom{\ddots} & & & \\
0 & 0 & 0 & 0 & & 0 & 0 & 0 \\
0 & 1 & 1 & 0 & & 0 & 0 & 0
\end{matrix} \\
\hline
\begin{matrix}
0 & 0 & 0 & & 0 & 1 & 1 & 0 \\
0 & 0 & 0 & & 0 & 0 & 0 & 0 \\
& & & & & & &  \\
& & & \vdots & & & & \\
& & & & & & & \\
& & & \hphantom{\ddots} & & & & \\
& & & & & & & \\
0 & 0 & 0 & & 0 & 0 & 0 & 0
\end{matrix}
&
\begin{matrix}
0 & 1 & 0 & 1 & & 0 & 0 & 0 \\
0 & 1 & 0 & 1 & & 0 & 0 & 0 \\
0 & 1 & 0 & 1 & & 0 & 0 & 0 \\
& \vdots & & & \ddots & & \vdots & \\
0 & 1 & 0 & 0 & & 1 & 0 & 0 \\
0 & 1 & 0 & 0 & & 0 & 1 & 0 \\
0 & 1 & 0 & 0 & & 0 & 0 & 1 \\
1 & 0 & 0 & 0 & & 0 & 0 & 1
\end{matrix}
\end{array}\ \right] \]
\captionsetup{labelfont=bf}
\caption{General form of the $(2n+4)$-by-$(2n+4)$ adjacency matrix $A_n$.}
\label{fig:adj-matrix}
\end{figure}

The spectrum of $M_n := M_{\kappa_n}$ is just the spectrum of $A_n$ scaled by $(2(1+\kappa))^{-1}$, so we may focus our analysis on $A_n$. For each $n$, let \[ V_n = \subspan_{\C}\set{\mathds{1}_{R_i}}{ 1 \leq i \leq 2n+4}, \] where $\{R_i\}$ is the Markov partition for $T_n$. 

\section*{Spectral Properties of $A_n$}

Observe that the Markov partition for $T_n$ is symmetric about $0$. Moreover, observe that for any $\kappa$, $T_{\kappa}$ is odd: $T_{\kappa}(-x) = -T_{\kappa}(x)$. In particular, if $\psi(x) = -x$, then $T_n\circ\psi = \psi\circ T_n$. The map $\psi$ has a Perron-Frobenius operator, $P_\psi$, and the commutation relation says that $P_nP_\psi = P_\psi P_n$. We also see that the map $\psi$ is Markov on the same partition as $T_n$, and noting the symmetry of this partition we obtain $\psi(R_i) = R_{2n+5-i}$. Thus the action of $P_\psi$ on $V_n$ is represented by the matrix $J_n$, as shown in Figure \ref{fig:Jn-matrix}, and we have $M_nJ_n = J_nM_n$, so also $A_nJ_n = J_nA_n$. It is also clear that because $\psi^2 = \id$, we have $J_n^2 = I$. Because $J_n^{-1} = J_n$, we have $J_nA_nJ_n = A_n$. Left-multiplication by $J_n$ reverses the order of the rows and right-multiplication by $J_n$ reverses the order of the columns, so the combination of both of them is performing a half-circle rotation of the matrix; we thus have independent verification of the half-circle rotational symmetry of $A_n$, which could be seen from Figure \ref{fig:adj-matrix}.

\begin{figure}
\[ \begin{bmatrix}
0 & 0 & 0 &           & 0 & 0 & 1 \\
0 & 0 & 0 & \dots     & 0 & 1 & 0 \\
0 & 0 & 0 &           & 1 & 0 & 0 \\
  & \vdots & & \iddots & & \vdots & \\
0 & 0 & 1 &           & 0 & 0 & 0 \\
0 & 1 & 0 & \dots     & 0 & 0 & 0 \\
1 & 0 & 0 &           & 0 & 0 & 0
\end{bmatrix} \]
\captionsetup{labelfont=bf}
\caption{The $(2n+4)$-by-$(2n+4)$ matrix $J_n$.}
\label{fig:Jn-matrix}
\end{figure}

Recall (see Theorem 1.3.19 in \cite{horn-johnson}) that if two diagonalizable matrices commute, then they are simultaneously diagonalizable, meaning that there is a shared basis of eigenvectors for the matrices. Also note that if two $r$-by-$r$ matrices $A$ and $B$ commute, then $\C^{2n+4}$ becomes a left-$\C[x,y]$-module, by setting $p(x,y)v := p(A,B)v$ for polynomials $p(x,y) \in \C[x,y]$ and $v\in \C^{2n+4}$. We will now use these facts to find many spectral properties of $A_n$, using its relation with $J_n$; note that we will find the spectral data of the entire sequence of $A_n$ all at once! We use Axler's approach to determinant-free linear algebra \cite{axler} and the practical implementation of those ideas by McWorter and Meyers \cite{no-dets}. Moreover, we make significant use of the underlying map $T_n$ to read off the algebraic relationships satisfied by $A_n$ and $J_n$ without doing a single matrix computation. We therefore reduce much of the study of the Perron-Frobenius operators to matrices that are easily studied by looking directly at the underlying maps.

\begin{lemma}
\label{lem:two-var-min-poly}
We have $A_n(A_n^{n+1}-2A_n^n-2J_n) = 0$.
\end{lemma}

\begin{proof}
Zooming in on the interval $[-1/2,0]$, as in Figure \ref{fig:markov-zoom}, we can identify the intervals $R_{n-1}$ through $R_{n+2}$. The interval $R_{n-1}$ is the interval immediately to the left of the left zero of $T_n$ in $[-1,0]$; the interval $R_n$ is the left branch of the leaking from $[-1,0]$ to $[0,1]$; the interval $R_{n+1}$ is the large interval $(-1/2,-\kappa_n)$; and the interval $R_{n+2}$ is the interval $(-\kappa_n,0)$.

\begin{figure}[tbp]
\centering
\includegraphics[width=0.75\textwidth]{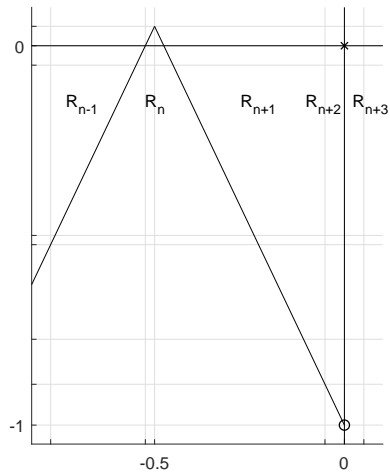}
\caption{A zoomed-in look at the Markov partition for $T_n$ in $[-1,0]$.} % % % % % % % % % %
\label{fig:markov-zoom}
\end{figure}

Looking at the map $T_n$ and using the Markov partition, we see that for $i \leq n-1$, the interval $R_{n+2}$ is mapped to $R_1$ and is subsequently expanded to the interval $(-1,r_i)$ in $i$ total steps, which is represented by \[ A_n^i e_{n+2} = e_1 + \dots + e_i. \] In the case of $i=n-1$, we have $T_n^{n-1}(R_{n+2}) = (-1,r_{n-1})$, and because $r_{n-1}$ is the left zero for $T_n$ in $[-1,0]$, we have $T_n^n(R_{n+2}) = (-1,0)$, which is represented by \[ A_n^n(e_{n+2}) = e_1 + \dots + e_{n+2}. \] Then, we clearly have $T_n(-1,0) = (-1,\kappa_n)$, so that because $T_n$ is (except at $-1/2$) $2$-to-$1$ on $[-1,0]$, we have \[ A_n^{n+1} e_{n+2} = 2(e_1 + \dots e_{n+3}) = 2A_n^n e_{n+2} + 2J_ne_{n+2}, \] where we used $J_ne_i = e_{2n+5-i}$. We rearrange this to $(A_n^{n+1}-2A_n^n-2J_n)e_{n+2} = 0$. Because $A_n$ and $J_n$ commute, we see that for any polynomial $p \in \C[x,y]$, we have \[ (A_n^{n+1} - 2A_n^n - 2J_n)p(A_n,J_n)e_{n+2} = 0. \]

Now, by the equations for $A_n^i e_{n+2}$, the fact that $J_n e_{n+2} = e_{n+3}$, and the fact that $A_n$ and $J_n$ commute, we see that 
\begin{align*} 
\subspan_{\C} &\set{p(A_n,J_n)e_{n+2}}{p\in \C[x,y]} \\
& = \subspan_{\C}\left\{ e_1,\dots,e_{n-1},(e_n+e_{n+1}),e_{n+2},\right. \\
& \qquad \qquad \qquad \left. e_{n+3},(e_{n+4}+e_{n+5}),e_{n+6},\dots,e_{2n+4}\right\}.
\end{align*} 
We can see that acting on the vector $e_{n+2}$ by $A_n$ and $J_n$ does not separate $e_n$ and $e_{n+1}$, or $e_{n+4}$ and $e_{n+5}$, by observing that any image of an $R_i$ either does not intersect $R_n$ and $R_{n+1}$ or covers both (and similarly for $R_{n+4}$ and $R_{n+5}$). Moreover, this subspace does not contain the vectors $v_1 = e_1 + \dots + e_n - (e_{n+1} + e_{n+2})$ and $v_2 = J_nv_1$. These are two linearly independent vectors that both lie in the kernel of $A_n$; to see they lie in the kernel, observe that the two vectors are representing $\mathds{1}_{(-1,-1/2)} - \mathds{1}_{(-1/2,0)}$ and the reflection $\mathds{1}_{(1/2,1)} - \mathds{1}_{(0,1/2)}$, and the intervals stretch to the same image. All together, we now have a basis for $\C^{2n+4}$, every element of which is annihilated by $A_n(A_n^{n+1}-2A_n^n-2J_n)$, and hence $A_n(A_n^{n+1}-2A_n^n-2J_n) = 0$.
\end{proof}

Let $E^+$ and $E^-$ be the subspaces of symmetric and antisymmetric vectors in $\C^{2n+4}$, respectively. 

\begin{lemma}
\label{lem:Jn-diag}
For all $n\geq 1$, $J_n$ is diagonalizable, with eigenspace $E^+$ corresponding to the eigenvalue $1$ and eigenspace $E^-$ corresponding to the eigenvalue $-1$. Moreover, the eigenspaces $E^{\pm}$ are $A_n$-invariant. 
\end{lemma}

\begin{proof}
We have $J_n^2 = I$, so that $(J_n-I)(J_n+I) = 0$. Since $J_n\pm I \ne 0$, we see that the minimal polynomial of $J_n$ is \[ m_{J_n}(x) = x^2-1 = (x-1)(x+1), \] and so $J_n$ is diagonalizable (because the minimal polynomial is separable), with eigenvalues $\pm 1$. The projections onto the eigenspaces $E_{+1}$ and $E_{-1}$ are given by \[ \frac{J_n+I}{1+1} = \tfrac{1}{2}(I+J_n), \qquad \frac{J_n-I}{-1-1} = \tfrac{1}{2}(I-J_n), \] respectively, by normalizing the factor of the minimal polynomial that does not annihilate the appropriate space. This immediately shows that $E_{+1} = E^+$ and $E_{-1} = E^-$. Finally, $A_n$ and $J_n$ commute, so for $s\in E^+$ and $a\in E^-$ we have \[ J_nA_ns = A_nJ_ns = A_ns, \qquad J_nA_na = A_nJ_na = -A_na, \] thus showing that $E^{\pm}$ are $A_n$-invariant.
\end{proof}

For notation, for all $n\geq 1$ let $f_n(x) = x^n(x-2) - 2$, $g_n(x) = x^n(x-2) + 2$, and $h_n(x,y) = x^n(x-2)-2y$. 

\begin{lemma}
\label{lem:poly-properties}
The polynomials $f_n$ and $g_n$ are irreducible over $\Q[x]$, separable with no roots at zero, and do not share any roots. 
\end{lemma}

\begin{proof}
For irreducibility, apply Eisenstein's Criterion with $p=2$ in both cases, followed by Gauss's Lemma. Since $\Q$ is characteristic zero, $f_n$ and $g_n$ are both separable. Clearly $0$ is not a root of either polynomial, and since $f_n(x) = g_n(x)-4$, the two polynomials cannot share any roots.
\end{proof}

\begin{proposition}
\label{prop:adj-spec-theory}
Let $n\geq 1$. We have:
\begin{enumerate}
\item the kernel of $A_n$ is $\ker(A_n) = \subspan_{\C}\{v_1+J_nv_1\} \oplus \subspan_{\C}\{v_1-J_nv_1\}$, for $v_1 = e_1 + \dots + e_n - (e_{n+1} + e_{n+2})$;
\item for $s\in E^+$, $h(A_n,J_n)s = f_n(A_n)s$, and the minimal polynomial of $A_n$ restricted to $E^+$ is $xf_n(x)$;
\item for $a\in E^-$, $h(A_n,J_n)a = g_n(A_n)a$, and the minimal polynomial of $A_n$ restricted to $E^-$ is $xg_n(x)$;
\item the minimal polynomial of $A_n$ is \[ m_{A_n}(x) = xf_n(x)g_n(x) = x(x^{2n+2} - 4x^{2n+1} + 4x^{2n} - 4); \]
\item the characteristic polynomial of $A_n$ is $\chi_{A_n}(x) = x m_{A_n}(x)$;
\item $A_n$ is diagonalizable over $\C$, with all eigenvectors corresponding to roots of $f_n$ being symmetric and all eigenvectors corresponding to roots of $g_n$ being antisymmetric.
\end{enumerate}
\end{proposition}

\begin{proof}
First, we have already seen (in the proof of Lemma \ref{lem:two-var-min-poly}) that $v_1$ and $J_nv_1$ form a basis for the kernel of $A_n$, so $v_1+J_nv_1$ and $v_1-J_nv_1$ also form a basis of the kernel of $A_n$ (one that conveniently splits into a symmetric and antisymmetric part).

Observe that $J_n$ restricted to $E^{\pm}$ is $\pm I$. Thus, we have, for $s\in E^+$ and $a\in E^-$: 
\begin{gather*}
h_n(A_n,J_n)s = h_n(A_n,I)s = (A_n^{n+1} - 2A_n^n - 2I)s = f_n(A_n)s, \\
h_n(A_n,J_n)a = h_n(A_n,-I)a = (A_n^{n+1} - 2A_n^n + 2I)a = g_n(A_n)a. 
\end{gather*}
Thus the minimal polynomial for $A_n$ restricted to $E^+$ and to $E^-$ are factors of $xf_n(x)$ and $xg_n(x)$, respectively, because $A_nh_n(A_n,J_n) = 0$. Then, since 
\begin{gather*}
A_n(e_{n+2} \pm J_ne_{n+2}) \ne 0, \\
f_n(A_n)(v_1+J_nv_1) \ne 0, \qquad g_n(A_n)(v_1-J_nv_1) \ne 0,
\end{gather*}
we see that the minimal polynomials for $A_n$ restricted to $E^+$ and $E^-$ must be exactly equal to $xf_n(x)$ and $xg_n(x)$. Then the minimal polynomial for $A_n$ on $\C^{2n+4} = E^+ \oplus E^-$ is the lowest common multiple of the minimal polynomials for $A_n$ on each subspace $E^{\pm}$, which means that $m_{A_n}(x) = xf_n(x)g_n(x)$ (as $f_n$ and $g_n$ share no roots). The degree of $m_{A_n}(x)$ is $2n+3$, but we know that the kernel is two-dimensional, so the characteristic polynomial must be $\chi_{A_n}(x) = xm_{A_n}(x)$, since the degree of $\chi_{A_n}(x)$ is exactly $2n+4$.

Lastly, the minimal polynomial is separable, by Lemma \ref{lem:poly-properties}, so we see that $A_n$ is diagonalizable. Since $A_n$ and $J_n$ commute, there is a basis of shared eigenvectors for $A_n$ and $J_n$. If $v$ is a non-kernel eigenvector for $A_n$, then as an eigenvector for $J_n$ it is either an element of $E^+$ or $E^-$; when $A_nv = \lambda v$ for $\lambda$ a root of $f_n$, then $v\in E^+$ since the minimal polynomial for $E^+$ is $xf_n(x)$ and $\lambda$ is not a root of $g_n$. Similarly, an eigenvector corresponding to a root of $g_n$ is an element of $E^-$. The proof is complete.
\end{proof}

We see that $A_n$ has zero as an eigenvalue with multiplicity two, and the non-zero eigenvalues of $A_n$ are the roots of the polynomials $f_n$ and $g_n$. We will now show that for $n\geq 5$, both $f_n$ and $g_n$ have a real root near $2$, one larger and one smaller respectively, and that all of the other roots are found near the unit circle. Figure \ref{fig:roots-fn-gn} shows the roots of $m_{A_n}$ for four different $n$.

\begin{figure}[tbp]
\centering
\begin{subfigure}{0.475\textwidth}
\centering
\includegraphics[width=\textwidth]{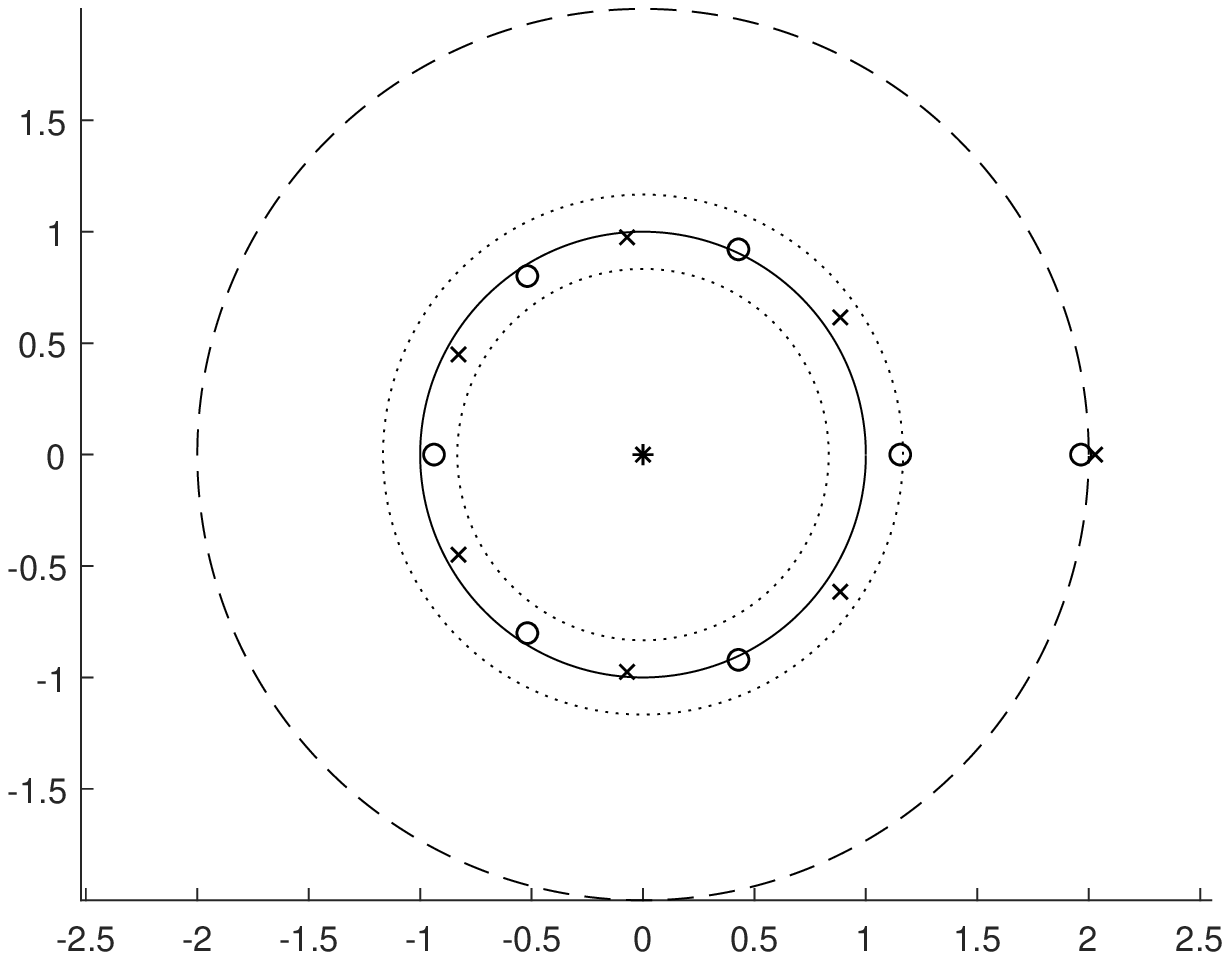} 
\caption{n=6}
\label{fig:roots-fn-gn-6}
\end{subfigure}
\hfill
\begin{subfigure}{0.475\textwidth}
\centering
\includegraphics[width=\textwidth]{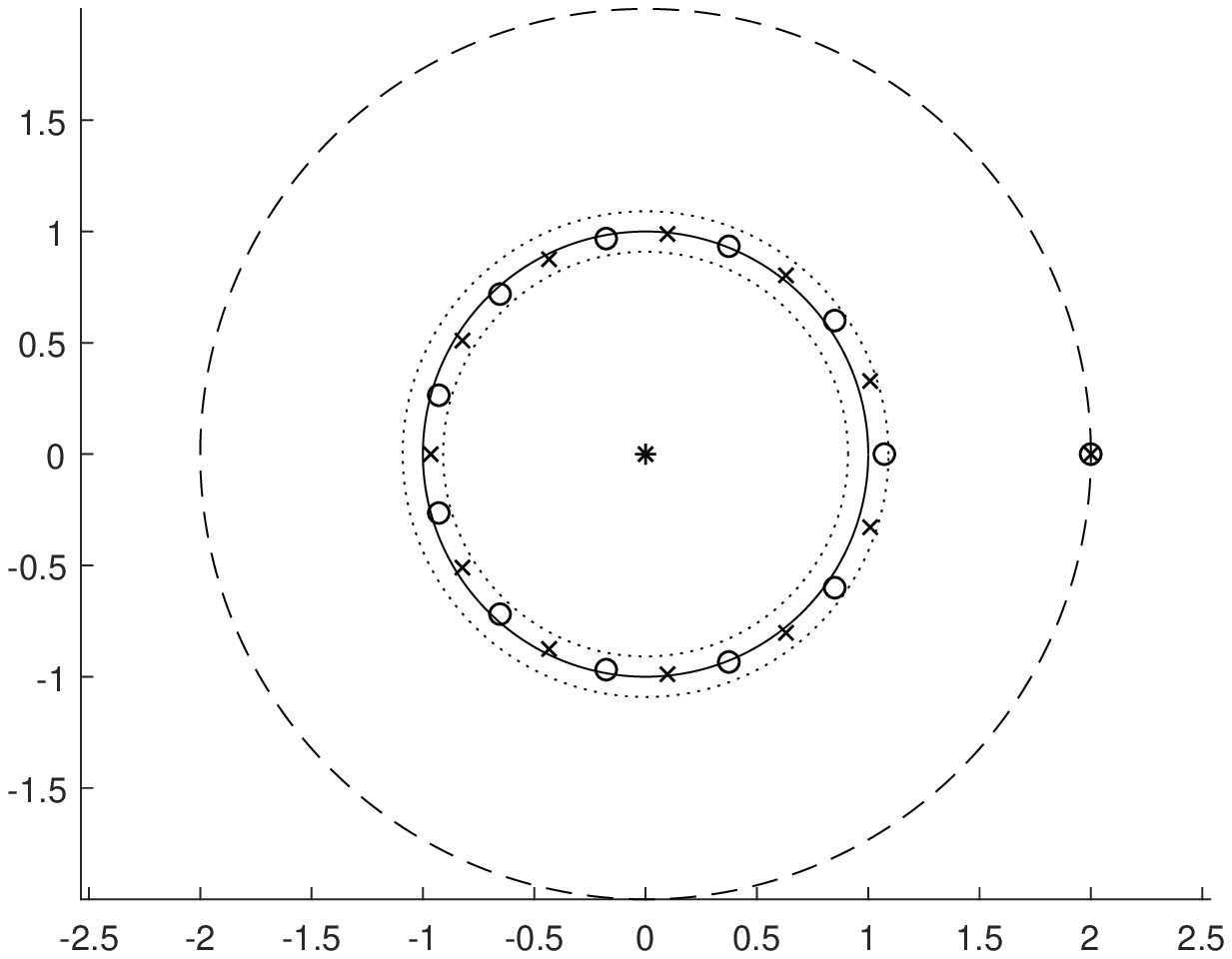} 
\caption{n=11}
\label{fig:roots-fn-gn-11}
\end{subfigure}

\vskip \baselineskip 

\begin{subfigure}{0.475\textwidth}
\centering
\includegraphics[width=\textwidth]{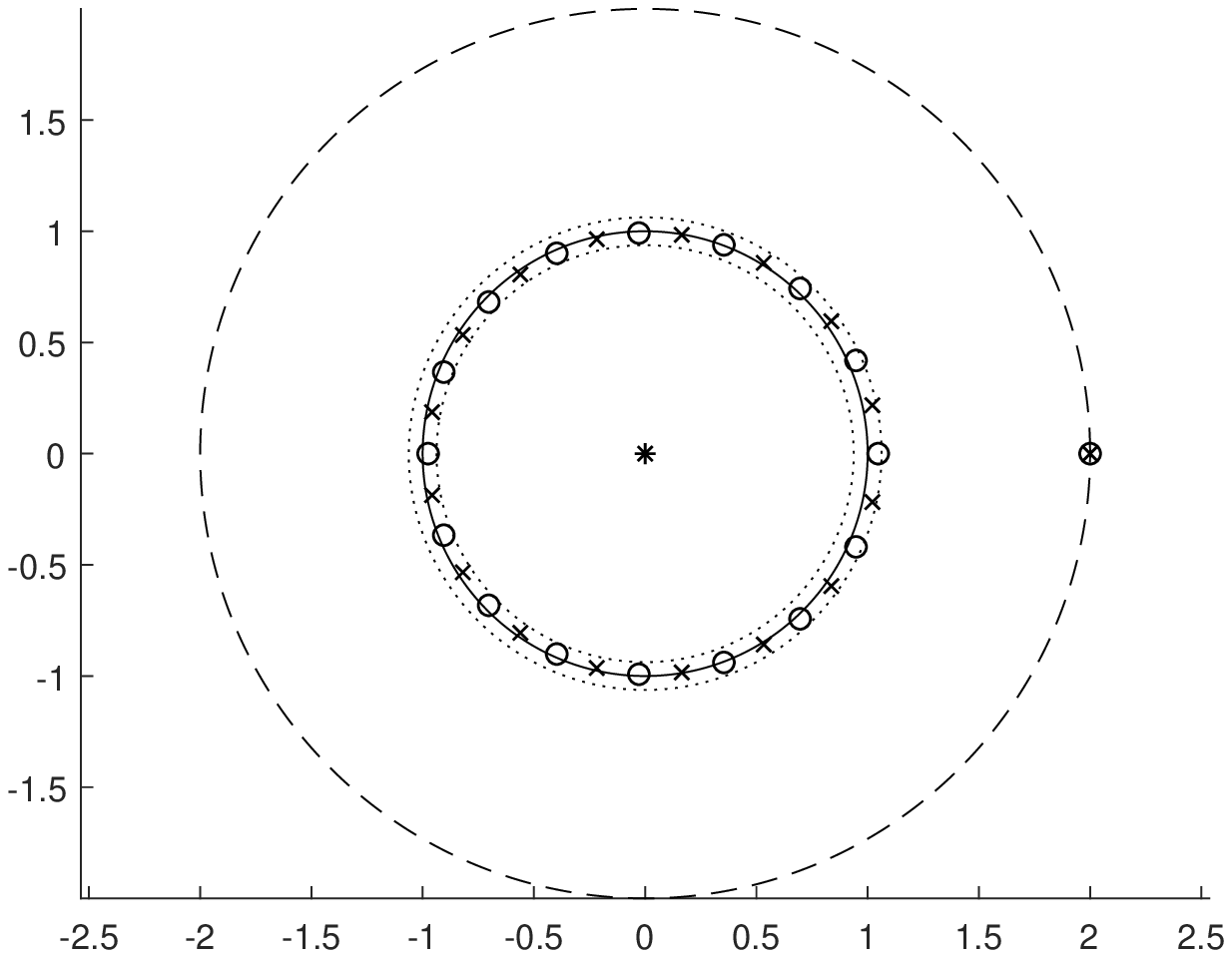}
\caption{n=16}
\label{fig:roots-fn-gn-16}
\end{subfigure}
\hfill
\begin{subfigure}{0.475\textwidth}
\centering
\includegraphics[width=\textwidth]{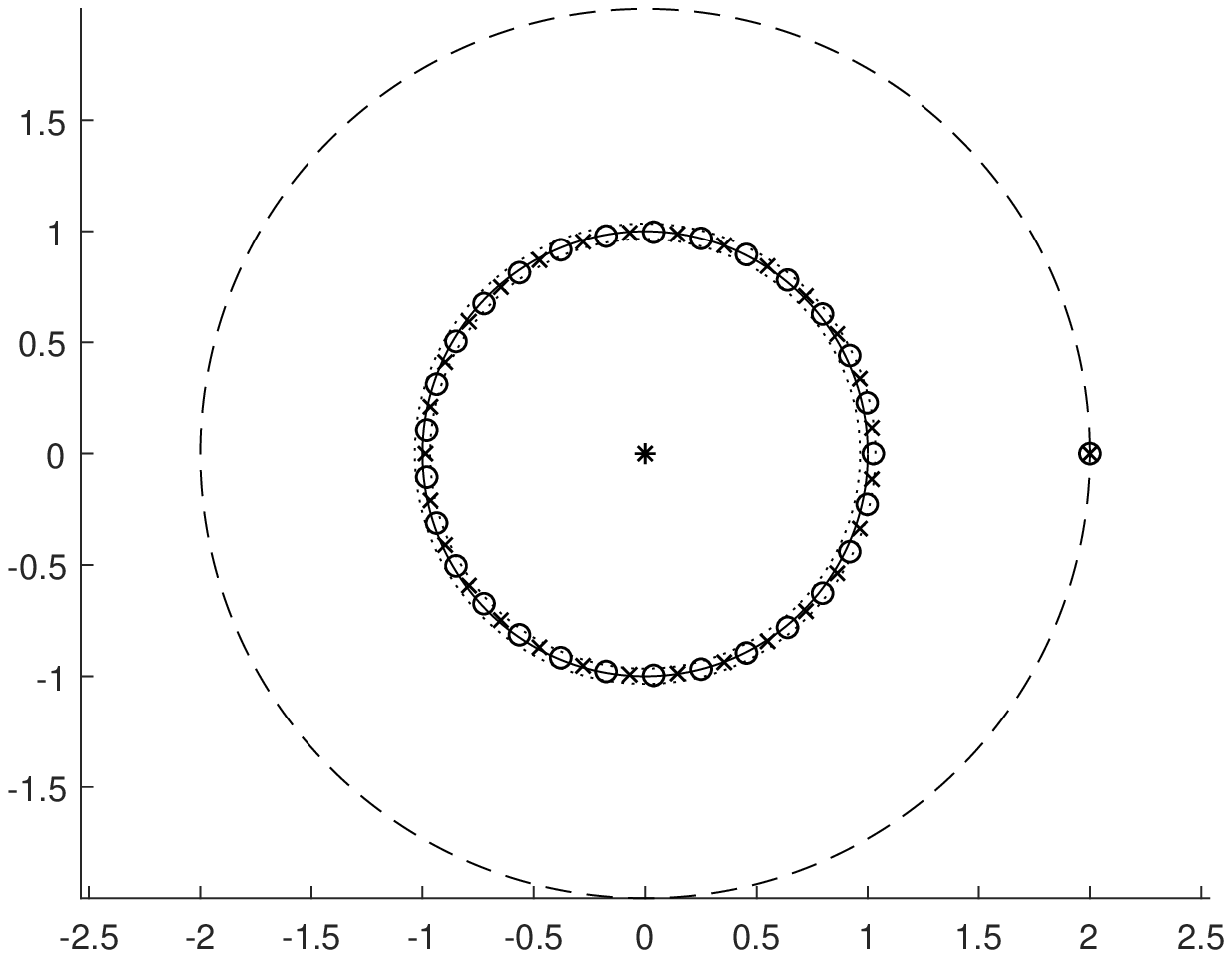}
\caption{n=29}
\label{fig:roots-fn-gn-29}
\end{subfigure}
%\captionsetup{width=0.85\textwidth,labelfont=bf}
\caption{Pictures of the roots of $f_n$ and $g_n$ for different values of $n$; roots of $f_n$ are marked with crosses, roots of $g_n$ are marked with circles, and the origin is marked with an asterisk (where $A_n$ has a double eigenvalue). The circle of radius $2$ is a dashed line, the unit circle is a solid line, and the circles with radius $1\pm n^{-1}$ are dotted lines.}
\label{fig:roots-fn-gn}
\end{figure}

\begin{proposition}
\label{prop:e-vals-markov}
The polynomial $f_n$ has the spectral radius of $A_n$, $2+2\kappa_n$, as a root, and $\kappa_n \sim \frac{1}{2^n}$. For $n\geq 5$, the polynomial $g_n$ has a real root at $2-2r_n < 2$, with $r_n \sim \kappa_n$. For all $n\geq 1$, all other roots of $f_n$ and $g_n$ are outside the circle of radius $1-n^{-1}$, and for all $n\geq 6$, all other roots of $f_n$ and $g_n$ are inside the circle of radius $1+n^{-1}$.
\end{proposition}

\begin{proof}
Observe that substituting $2+2\kappa_n$ into $f_n$ yields:
\begin{align*}
f_n(2+2\kappa_n) & = (2+2\kappa_n)^n(2+2\kappa_n-2) - 2 \\
& = 2\left( (2+2\kappa_n)^n\kappa_n - 1 \right) = 0,
\end{align*}
by the definition of $\kappa_n$. To see roughly how big $\kappa_n$ is, observe that \[ \frac{1}{2^n} > \frac{1}{(2+2\kappa_n)^n} = \kappa_n > \frac{1}{\left( 2+\frac{2}{2^n} \right)^n} = \frac{1}{2^n}\cdot\frac{1}{\left( 1+\frac{1}{2^n} \right)^n}. \] Since $(1+2^{-n})^n$ converges to $1$ as $n$ tends to infinity, we see that $\kappa_n \sim 2^{-n}$. 

Next, observe that applying $g_n$ to $2(1-2^{-n})$ yields 
\begin{align*} 
g_n(2(1-2^{-n})) & = 2^n\left(1-\frac{1}{2^n}\right)^n\left(2-2\cdot \frac{1}{2^n} -2\right) + 2 \\
& = -2\left( 1-\frac{1}{2^n} \right)^n + 2 > 0. 
\end{align*} 
Then, evaluate $g_n$ at $2(1-(1+2n/2^{-n})\cdot 2^{-n})$ and use Bernoulli's inequality (used here in the form $(1-x)^n > 1-nx$ for $0<x<1$ and $n\geq 1$):
\begin{align*} 
g_n\left( 2\left( 1-\left(\frac{1}{2^n} + \frac{2n}{4^n} \right) \right) \right) & = -2^{n+1}\left( 1-\left(\frac{1}{2^n} + \frac{2n}{4^n} \right) \right)^n \left(\frac{1}{2^n} + \frac{2n}{4^n} \right) + 2 \\
& \leq 2\left(1 - \left( 1-\left(\frac{n}{2^n} + \frac{2n^2}{4^n} \right) \right)\left( 1+\frac{2n}{2^n} \right) \right) \\
& = \frac{2n}{2^n}\left( -1 + \frac{4n}{2^n} + \frac{4n^2}{4^n} \right).
\end{align*} 
The quantity inside the parentheses is decreasing for $n \geq 2$ and is negative for $n\geq 5$, so by continuity of $g_n$ there exists a root $2-2r_n$ of $g_n$ for $n\geq 5$, where $2^{-n} < r_n < 2^{-n} + 2n4^{-n}$. Thus $r_n \sim 2^{-n} \sim \kappa_n$.

Lastly, we use Rouch\'e's Theorem to estimate the other roots of $f_n$ and $g_n$. Set $a(z) = 2$ and $b_n(z) = z^n(z-2)$. For $\abs{z} = 1-n^{-1}$, we have $\abs{z-2} \leq \abs{z}+2 = 3+n^{-1}$ and hence (noting that $(1-x)^n < (1+x)^{-n}$ for $0 < x < 1$ and $n \geq 1$):
\begin{align*}
\abs{b_n(z)} & = \abs{z}^n\abs{z-2} \leq \left( 1-\frac{1}{n} \right)^n\left( 3+\frac{1}{n} \right) \\
& < \frac{3 +1/n}{\left( 1+1/n \right)^n} \leq \frac{3+1}{1+\frac{n}{n}} = 2 = \abs{a(z)},
\end{align*}
where the last inequality came from the first two terms of the Binomial expansion. We apply Rouch\'e's Theorem to see that $a(z)$ and $a(z)\pm b_n(z) = g_n(z),-f_n(z)$ (so also $f_n(z)$) have the same number of roots inside $\abs{z} = 1-n^{-1}$: none, because $a(z)$ is constant and therefore has no roots. On the other hand, for $\abs{z} = 1+n^{-1}$, we have $\abs{z-2} \geq 2-\abs{z} = 1-n^{-1}$, so again using the Binomial expansion (three terms, this time), we get:
\begin{align*}
\abs{b_n(z)} & = \abs{z}^n\abs{z-2} \geq \left( 1+\frac{1}{n} \right)^n\left( 1-\frac{1}{n} \right) \\
& > \left( 1+ \frac{n}{n} + \frac{n(n-1)}{2n^2} \right)\left( 1-\frac{1}{n} \right) \\
& = \left( \frac{5}{2} - \frac{1}{2n} \right)\left( 1-\frac{1}{n} \right).
\end{align*}
This last quantity is clearly increasing, and for $n\geq 6$ it is larger than $2 = \abs{a(z)}$. Thus, for $n\geq 6$, Rouch\'e's Theorem says that $b_n(z)$ and $b_n(z)\pm a(z) = g_n(z),f_n(z)$ have the same number of roots inside $\abs{z} = 1 + n^{-1}$: $n$, because the $n+1$ roots of $b_n(z)$ are $0$ with multiplicity $n$ and $2$ with multiplicity $1$, and $2$ is certainly outside the circle of radius $1+n^{-1}$ if $n\geq 6$.
\end{proof}

\begin{corollary}
\label{cor:spec-radius}
For all $n\geq 1$, the spectral radius of $A_n$ is $2(1+\kappa_n)$, and so the spectral radius of $M_n$ is $1$.
\end{corollary}

\begin{proof}
For $n\leq 5$, one may use a computer to show that the only root of $m_{A_n}(x)$ at least of magnitude $2$ is $2+2\kappa_n$. For $n\geq 6$, we use Proposition \ref{prop:e-vals-markov} to conclude that the largest eigenvalue is $2+2\kappa_n$. The spectrum of $M_n$ is simply the spectrum of $A_n$ scaled by $(2+2\kappa_n)^{-1}$, so the spectral radius of $M_n$ is $1$.
\end{proof}

In \cite{horan-pf-thm}, the application of Proposition \ref{prop:e-vals-markov} is to provide a sharpness result for a general estimate on the exponential mixing rate for a class of non-autonomous dynamical systems that are perturbations of the map $T_{0}$. This computation is reproduced as the following Corollary, and describes how the second-largest eigenvalue $(2-2r_n)(2+2\kappa_n)^{-1}$ for $M_n$ approaches $1$ as $n$ tends to infinity.

\begin{corollary}
\label{cor:sec-eval}
The second largest eigenvalue for $M_n$, and hence for $P_n$, is asymptotically equivalent to $1-2\kappa_n$.
\end{corollary}

\begin{proof}
For $n \geq 6$, we know that the second largest eigenvalue in modulus for $A_n$ is $2-2r_n$, so the second largest eigenvalue in modulus for $M_n$, and thus $P_n$ (by Lemma \ref{lem:markov-op}), is $(2-2r_n)(2+2\kappa_n)^{-1}$. By Proposition \ref{prop:e-vals-markov}, we have $r_n = \kappa_n + o(\kappa_n)$. Thus we have:
\begin{align*} 
\frac{2-2r_n}{2+2\kappa_n} & = (1-r_n)(1+\kappa_n)^{-1} = (1-\kappa_n + o(\kappa_n))(1-\kappa_n+o(\kappa_n)) \\
& = 1-2\kappa_n + o(\kappa_n). \qedhere
\end{align*}
\end{proof}

\begin{remark}
\label{rem:tensor}
We identified $J_n$ as a $(2n+4)$-by-$(2n+4)$ permutation matrix with ones along the anti-diagonal. To add to our knowledge about $J_n$, we can also identify $J_n$ as a $2$-by-$2$ flip, by using tensor products: $\C^{2n+4} \simeq E^+\otimes_{\C} \C^2$, and $J_n$ is the flip in the second coordinate.
\end{remark}

\begin{remark}
\label{rem:alg-geo}
In the proof of Proposition \ref{prop:adj-spec-theory}, we computed the minimal polynomial for $A_n$ by finding invariant subspaces and working with the restrictions of $A_n$ to those subspaces; the relation $A_n(A_n^{n+1} - 2A_n^n -2J_n) = 0$ reduced to the relations $A_n(A_n^{n+1} - 2A_n^n \mp 2I) = 0$ on the subspaces $E^{\pm}$, and these one-matrix relations yielded to standard techniques. However, we also had $J_n^2 = I$, and we could consider the simultaneous equations 
\begin{gather*}
A_n h_n(A_n,J_n) = A_n(A_n^{n+1} - 2A_n^n -2J_n) = 0, \\
J_n^2 - I = 0.
\end{gather*}
Is it possible to take an algebraic-geometric approach to finding the eigenvalues of $A_n$ and $J_n$ without reducing to the single-variable theory? The answer is yes! 

Briefly, by Hilbert's Nullstellensatz we see that the ideal of polynomials in two variables that vanish on the locus of $\{ xh_n(x,y),y^2-1\}$ is the same as the radical of the ideal $\cJ$ generated by $\{ xh_n(x,y),y^2-1\}$. Even better, one can show that $\cJ$ is actually radical, and that $\cJ$ is moreover equal to the ideal of polynomials that vanish when evaluated at $(A_n,J_n)$. Finally, the polynomial $m_{A_n}(x) = xf_n(x)g_n(x)$ is shown to be an element of $\cJ$, so we get that $A_n$ is diagonalizable in the same way as before; this means $\cJ$ is also equal to the ideal of polynomials that vanish on the pairs of eigenvalues $(\lambda,\mu)$ of $A_n$ and $J_n$. Hence the pairs of eigenvalues are exactly the locus of $\{ xh_n(x,y),y^2-1\}$, instead of just a subset, and solving for the roots of the two polynomials simultaneously yields the eigenvalues of $A_n$ and the anti/symmetric breakdown. This abstract perspective is another way to see the problem, though our initial proof was much less high-tech. The subsequent computations are not affected by the change.
\end{remark}

\section*{Spectral Properties of a Related System}

We may use our knowledge of $T_{\kappa_n}$ to study a related system. Define $\pi : [-1,1] \to [0,1]$ by $\pi(x) = \abs{x}$ and set $\tilde{T}_{\kappa} := \pi\circ T_{\kappa} = \abs{T_{\kappa}} : [0,1] \to [0,1]$, as depicted in Figure \ref{fig:factor-map}. We can ask the same questions about this map: does it have an invariant density? Is it mixing, and if so with what rate? Instead of repeating all of our work, however, we can use the relationship between $T_{\kappa}$ and $\tilde{T}_{\kappa}$ and the information about $T_{\kappa}$ to answer these questions, again by reducing the computations to painless matrix relations.

\begin{figure}[tbp]
\centering
\includegraphics[width=0.75\textwidth]{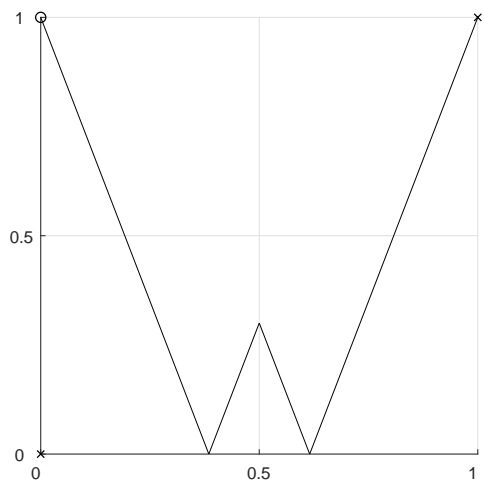}
\captionsetup{width=0.75\textwidth,labelfont=bf}
\caption{The map $\tilde{T}_{\kappa}$, for $\kappa = 0.3$.}
\label{fig:factor-map}
\end{figure}

Observe that because $T_{\kappa}$ is odd, for any $x\in [-1,1]$ we have that $T_{\kappa}$ maps $\{\pm x\}$ to $\{\pm T_{\kappa}(x)\}$. Looking at $\{\pm x\}$ as an equivalence class under $x\sim -x$, we see that the map $\pi$ defined above collapses each class to a single point in $[0,1]$: the shared absolute value of the elements of the class. From the definition of $\tilde{T}_{\kappa}$, then, it is clear that $\pi \circ T_{\kappa} = \tilde{T}_{\kappa} \circ \pi$.

In addition, for each $n\geq 1$, $\tilde{T}_n := \tilde{T}_{\kappa_n}$ is still Markov. The Markov partition is not quite the same as the partition on $[0,1]$ for $T_n$; the map is no longer monotonic on just two intervals, but rather four intervals. However, we can easily guess a partition; the interval $R_{n+4}$ should be split in two. Note that $1-\kappa_1 = \tfrac{1}{2}+\frac{\kappa_1}{2(1+\kappa_1)}$ and for all $n\geq 2$, $\tilde{T}^{n-1}_n(\kappa_n) = \tfrac{1}{2}+\tfrac{\kappa_n}{2(1+\kappa_n)}$; this point is the zero of $\tilde{T}_n$ larger than $1/2$, and the symmetric point $\tfrac{1}{2}-\tfrac{\kappa_n}{2(1+\kappa_n)}$ is the zero smaller than $1/2$.

\begin{lemma}
\label{lem:Markov-part-factor}
The map $\tilde{T}_n$ is Markov for each $n\geq 1$. For $n=1$, the Markov partition is \[ \Big\{ \left( 0, \kappa_1 \right), \left( \kappa_1, \tfrac{1}{2} \right), \left( \tfrac{1}{2}, 1-\kappa_1 \right) \left( 1-\kappa_1, 1 \right) \Big\} \] and for $n \geq 2$, the Markov partition is 
\begin{gather*}
\Big\{ \left( 0, \kappa_n \right), \left( \kappa_n, \tfrac{1}{2}-\tfrac{\kappa_n}{2(1+\kappa_n)} \right), \left(\tfrac{1}{2}-\tfrac{\kappa_n}{2(1+\kappa_n)}, \tfrac{1}{2} \right), \left( \tfrac{1}{2},  \tfrac{1}{2}+\tfrac{\kappa_n}{2(1+\kappa_n)} \right) \Big\} \\
\cup \Big\{ \left( \tilde{T}_n^{i+1}(\kappa_n), \tilde{T}_n^i(\kappa) \right) \Big\}_{i=1}^{n-2} \cup \Big\{ \left( \tilde{T}_n(\kappa_n), 1 \right) \Big\}.
\end{gather*}
The Markov partition has, in all cases, $n+3$ intervals.
\end{lemma}

\begin{proof}
We again apply Lemma \ref{lem:markov-points}. For $n=1$, the Markov partition is simply the (interiors of the) intervals of monotonicity, since $\tilde{T}_1(\kappa_1) = 0$ and $\tilde{T}_1(1/2) = \kappa_1$; clearly, there are $4 = 1+3$ intervals. For $n\geq 2$, the point $\tfrac{1}{2} - \tfrac{\kappa_n}{2(1+\kappa_n)}$ is mapped to $0$, and the remainder of the points are just as in the case of $T_{\kappa_n}$. Because we have split one of the intervals in $[0,1]$ in two (but are only considering $[0,1]$, not $[-1,1]$), there are exactly $n+3$ elements in the Markov partition.
\end{proof}

The Markov partitions for $n=1,4$ are illustrated in Figure \ref{fig:Markov-part-factor-1-4}. We will denote the intervals in order left-to-right by $\{S_1\}_{1}^{n+3}$. For $n\geq 4$ the general form of the adjacency matrix $B_n$ is as in Figure \ref{fig:adj-matrix-factor}. For $n\leq 3$ some of the columns are combined. Observe that $B_n$ is almost identical to the bottom-right quadrant of $A_n$, with the exception of an extra column; this is expected, given how we modified the Markov partition by splitting $R_{n+4}$ into $S_2$ and $S_3$ while leaving the other intervals the same.

\begin{figure}[tbp]
\centering
\begin{subfigure}{0.475\textwidth}
\centering
\includegraphics[width=\textwidth]{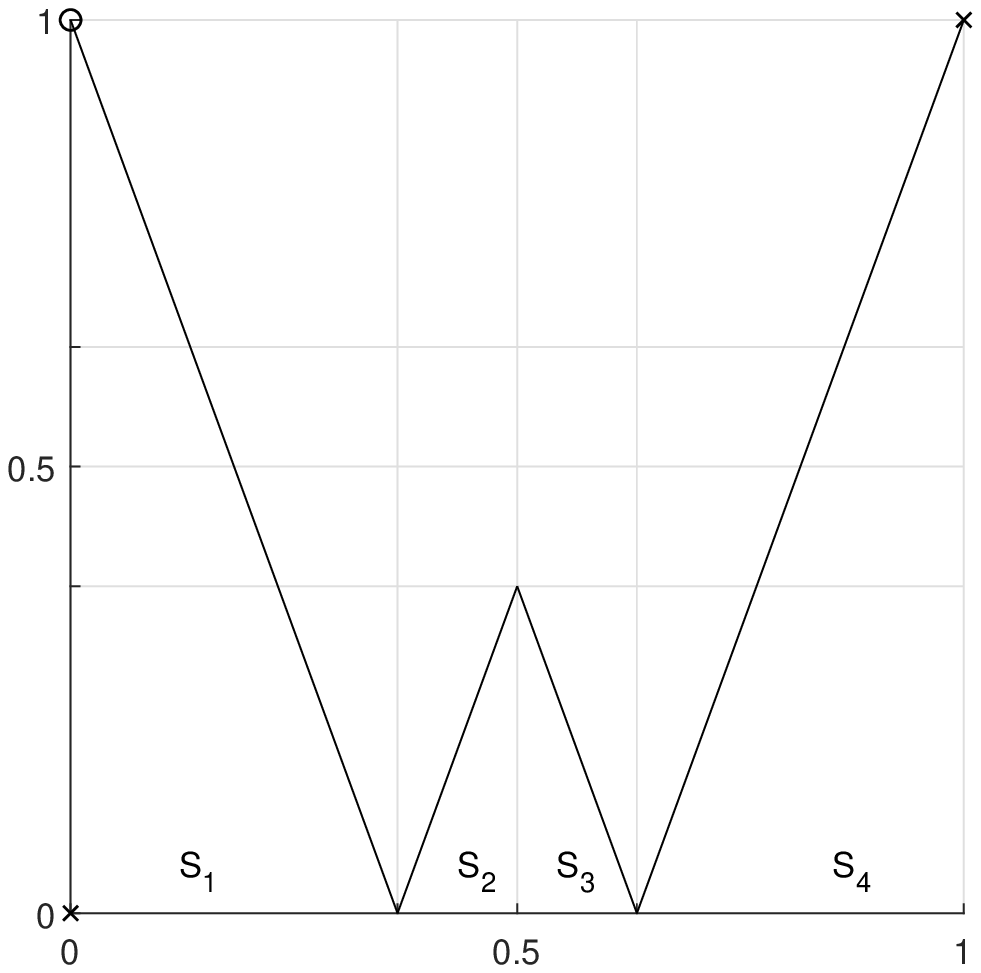}
\caption{n=1}
\label{fig:Markov-part-factor-1}
\end{subfigure}
\hfill
\begin{subfigure}{0.475\textwidth}
\centering
\includegraphics[width=\textwidth]{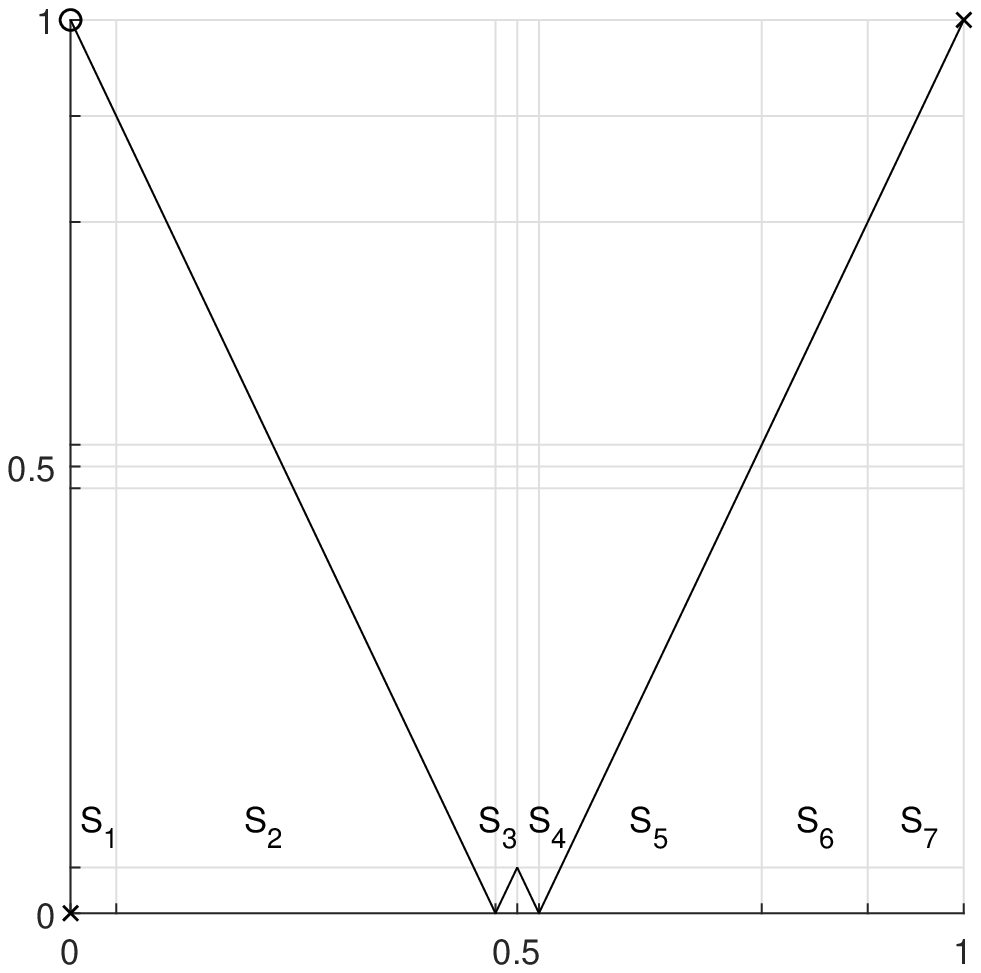}
\caption{n=4}
\label{fig:Markov-part-factor-4}
\end{subfigure}
\captionsetup{labelfont=bf}
\caption{Markov partitions for $\tilde{T}_n$, for $n=1,4$.}
\label{fig:Markov-part-factor-1-4}
\end{figure}

\begin{figure}[tbp]
\[ \begin{bmatrix}
0 & 1 & 1 & 1 & 1 & 0 & & 0 & 0 \\ 
0 & 1 & 0 & 0 & 1 & 0 & & 0 & 0 \\
0 & 1 & 0 & 0 & 1 & 0 & \dots & 0 & 0 \\
0 & 1 & 0 & 0 & 1 & 0 & & 0 & 0 \\
0 & 1 & 0 & 0 & 0 & 1 & & 0 & 0 \\
  &   & \vdots &&& & \ddots & & \\
0 & 1 & 0 & 0 & 0 & 0 & & 1 & 0 \\
0 & 1 & 0 & 0 & 0 & 0 & & 0 & 1 \\
1 & 0 & 0 & 0 & 0 & 0 & & 0 & 1
\end{bmatrix} \]
\captionsetup{labelfont=bf}
\caption{General form of the $(n+3)$-by-$(n+3)$ adjacency matrix $B_n$.}
\label{fig:adj-matrix-factor}
\end{figure}

We now compute the spectral data for $B_n$. Towards this goal, for notation let $\tilde{V}_n$ be the span of the functions $\{\mathds{1}_{S_i}\}_{i=1}^{n+3}$, and let $\tilde{\phi}: \tilde{V}_n \to \C^{n+3}$ be the isomorphism from Lemma \ref{lem:markov-op} with $\tilde{\phi}(\mathds{1}_{S_i}) = d_i$. Our proof will run through the action of $A_n$ on $E^+$; for each $i$ between $1$ and $n+2$, let $s_i = \frac{1}{2}(e_{n+3-i}+e_{n+2+i})$, and observe that $\{s_i\}_{i=1}^{n+2}$ is a basis for $E^+$. Then, call $C_n : \C^{n+2} \to \C^{n+2}$ the matrix representation of $A_n$ on $E^+ = \subspan_{\C}\{s_i\}_{i=1}^{n+2}$. Moreover, let $\iota$ be the matrix representations of the inclusion of $E^+$ into $\C^{n+3}$ by splitting up $s_2$ into $d_2 + d_3$, as pictured in Figure \ref{fig:iota-matrix}, so that $\iota(s_1) = d_1$, $\iota(s_2) = d_3+d_4$, and $\iota(s_{k}) = d_{k+1}$ for $k \geq 3$.

\begin{figure}
\centering
\[ \begin{bmatrix}
1 & 0 & 0 &   & 0 \\ 
0 & 1 & 0 &   & 0 \\
0 & 1 & 0 &   & 0 \\
0 & 0 & 1 &   & 0 \\
 & & & \ddots & 0 \\
0 & 0 & 0 &   & 1
\end{bmatrix} \]
\caption{The $(n+3)$-by-$(n+2)$ matrix $\iota$, representing the inclusion $E^+ \to \C^{n+3}$.}
\label{fig:iota-matrix}
\end{figure}

\begin{proposition}
\label{prop:adj-char-poly-factor}
Let $n\geq 1$. We have:
\begin{enumerate}
\item $\iota C_n = B_n \iota$, with $C_n$ as given in Figure \ref{fig:adj-matrix-E+};
\item the kernel of $B_n$ is $\ker(B_n) = \subspan_{\C}\{\big(d_3-d_4\big),\big(d_1+d_2 - (d_5+\dots+d_{n+3})\big)\}$;
\item the minimal polynomial for $B_n$ is $\mathrm{min}_{B_n}(x) = xf_n(x)$;
\item the characteristic polynomial for $B_n$ is $\mathrm{char}_{B_n}(x) = x^2f_n(x)$;
\item $B_n$ is diagonalizable, the spectral radius of $B_n$ is $2+2\kappa_n$, and all of the other eigenvalues of $B_n$ are zero or near the unit circle. The eigenvector corresponding to the spectral radius is $\iota(w)$, where $w\in E^+$ is the eigenvector corresponding to the spectral radius for $C_n$.
\end{enumerate}
\end{proposition}

\begin{figure}[tbp]
\[ \begin{bmatrix}
0 & 2 & 1 & 1 & & 0 & 0 & 0 \\
0 & 1 & 0 & 1 & & 0 & 0 & 0 \\
0 & 1 & 0 & 1 & & 0 & 0 & 0 \\
& \vdots & & & \ddots & & \vdots & \\
0 & 1 & 0 & 0 & & 1 & 0 & 0 \\
0 & 1 & 0 & 0 & & 0 & 1 & 0 \\
0 & 1 & 0 & 0 & & 0 & 0 & 1 \\
1 & 0 & 0 & 0 & & 0 & 0 & 1
\end{bmatrix} \]
\caption{The $(n+2)$-by-$(n+2)$ matrix $C_n$, representing the action of $A_n$ on $E^+$.}
\label{fig:adj-matrix-E+}
\end{figure}

\begin{proof}
First, note that the action of $A_n$ restricted to $E^+$ can be seen as identifying the vectors $e_{i}$ and $e_{2n+5-i}$ and looking at the action of $A_n$ on the vectors $e_{n+3}$ to $e_{2n+4}$, because $E^+ = \subspan_{\C}\{s_i\}_{i=1}^{n+2}$ with $s_i = \frac{1}{2}(e_{n+3-i}+e_{n+2+i})$. The columns of the matrix $A_n$ indicate the images under $T_n$ of the intervals $R_i$ for each $i$, and so the columns of the restriction $C_n$ indicate the images under $T_n$ of the intervals $R_{n+3}$ up to $R_{2n+4}$ under the identification of $R_i$ with $R_{2n+5-i}$ (considered with multiplicity). However, this is exactly what the columns of $B_n$ indicate, because $B_n$ is the adjacency matrix for $\tilde{T}_n$, taken with the refined partition $\{S_i\}_{i=1}^{n+3}$. Since $\iota$ represents the refinement of the partition, we have $\iota C_n = B_n \iota$. From this equality we can obtain the remainder of the results in Proposition \ref{prop:adj-char-poly-factor}.

Looking at $\tilde{T}_n$, it is clear that the kernel of $B_n$ is equal to $\subspan_{\C}\{\big(d_3-d_4\big),\big(d_1+d_2 - (d_5+\dots+d_{n+3})\big)\}$, because these two vectors represent the two facets of symmetry in $\tilde{T}_n$ (the symmetry in the long branches and the symmetry in the short branches).

To find the minimal polynomial for $B_n$, recall from Proposition \ref{prop:adj-spec-theory} that the minimal polynomial for $A_n$ restricted to $E^+$ is equal to $xf_n(x)$. Thus, we have \[ B_nf_n(B_n)\iota = \iota C_nf_n(C_n) = 0, \] since $C_n$ represents $A_n$ acting on $E^+$ and so satisfies the minimal polynomial. We also have that \[ \C^{n+3} = \Imag(\iota) \oplus \subspan_{\C}\{d_3-d_4\}, \] because $\iota$ is injective (with rank $n+2$) and $d_3-d_4$ is not in the image of $\iota$. Since $B_nf_n(B_n)$ annihilates both the image of $\iota$ and $d_3-d_4$ but $f_n(B_n)(d_3-d_4) = -2(d_3-d_4)$ and $B_n(d_1) \ne 0$, we see that the minimal polynomial of $B_n$ is $m_{B_n}(x) = xf_n(x)$. The characteristic polynomial for $B_n$ is $\chi_{B_n}(x) = x^2f_n(x)$, of course, because the kernel of $B_n$ is two-dimensional and the degree of $\chi_{B_n}(x)$ is $n+3$.

Finally, the minimal polynomial for $B_n$ is separable, so $B_n$ is diagonalizable. By Proposition \ref{prop:e-vals-markov}, the largest eigenvalue of $B_n$ is $2+2\kappa_n$, and all other eigenvalues of $B_n$ zero or near the unit circle (asymptotically). If $w\in E^+$ is the eigenvector corresponding to $2+2\kappa_n$ for $C_n$, then \[ B_n(\iota(w)) = \iota(C_nw) = (2+2\kappa_n)\iota(w), \] so $\iota(w)$ is the eigenvector for $B_n$ corresponding to $2+2\kappa_n$.
\end{proof}

Observe that $B_n$ shares no eigenvalues corresponding to the antisymmetric eigenvectors for $A_n$; this makes sense, since the map $\pi$ collapsed all of those vectors to $0$, and we are left with the the symmetric eigenvectors. It did, however, introduce a new kernel vector, by introducing a new aspect of symmetry.

In addition, note that we could not simply apply Lemma \ref{lem:markov-op} with the matrix $C_n$, because $\tilde{T}_n$ is not Markov with respect to the partition $\{R_i\}_{i=n+3}^{2n+4}$. However, the relationship between $C_n$ and $B_n$ allowed us to painlessly translate facts about $A_n$ (and $C_n$) into facts about $B_n$, the actual adjacency matrix for $\tilde{T}_n$.

\begin{corollary}
\label{cor:sec-eval-factor}
The second-largest eigenvalue of the Perron-Frobenius operator for $\tilde{T}_n$ has modulus at most $(1+n^{-1})(2+2\kappa_n)^{-1}$ (for $n\geq 6$), which is asymptotically equivalent to $\frac{1}{2}(1+n^{-1})$.
\end{corollary}

\begin{proof}
The spectral radius of $B_n$ is still $2+2\kappa_n$, by Corollary \ref{cor:spec-radius} and Proposition \ref{prop:adj-char-poly-factor}, so the spectral radius of the Perron-Frobenius operator for $\tilde{T}_n$ is $1$ and the second-largest eigenvalue has modulus at most $1+n^{-1}$ divided by $2+2\kappa_n$, using Proposition \ref{prop:adj-spec-theory} to get the upper bound. We then have (since $\kappa_n \sim 2^{-n}$):
\begin{align*}
\frac{1+n^{-1}}{2(1+\kappa_n)} & = \frac{1}{2}\left(1+\frac{1}{n}\right)(1 - \kappa_n + o(\kappa_n)) \\
& = \frac{1}{2}\left( 1 + \frac{1}{n} - \kappa_n + o(\kappa_n) \right) = \frac{1}{2}\left( 1+\frac{1}{n} + o\left(\frac{1}{n}\right) \right),
\end{align*}
which shows that $(1+n^{-1})(2+2\kappa_n)^{-1}$ is asymptotically equivalent to $\frac{1}{2}(1+n^{-1})$.
\end{proof}

\section*{Conclusion for Mixing Times}

At the beginning of this paper, we asked about mixing times and mixing rates for dynamical systems. We can now answer that question for our two systems, $T_n$ and $\tilde{T}_n$. For $T_n$, we have shown that the second-largest eigenvalue of the Perron-Frobenius $P_n$ is approximately $1-2\kappa_n$ (Corollary \ref{cor:sec-eval} and Lemma \ref{lem:markov-op}). Thus the mixing time for $T_n$ is, ignoring a scale factor, \[ \frac{1}{\abs{\log(1-2\kappa_n)}} \sim \frac{1}{2\kappa_n} \sim 2^{n-1}, \] using the fact that $\kappa_n \sim 2^{-n}$.

On the other hand, for $\tilde{T}_n$, Corollary \ref{cor:sec-eval-factor} says that the second-largest eigenvalue of the Perron-Frobenius operator has modulus at most $(1+n^{-1})(2+2\kappa_n)^{-1}$. By a similar computation, the mixing time for $\tilde{T}_n$ (with $n\geq 6$) is $O(1)$, which is much smaller than the mixing time for $T_n$. 

This result matches our intuition: for $T_n$, taking the perturbation to zero (or $n$ to infinity) leads to no mixing between the two halves, so the mixing time should tend to infinity, whereas for $\tilde{T}_n$, taking the perturbation to zero leads to a mixing tent map, and hence the mixing time should approach that for the unperturbed map. The difference in the orders of the mixing times indicates significant dynamical information about how these two systems are distinct, and we obtained this information by performing calculations with matrices (without touching the matrices themselves) and some analysis of roots of polynomials.

\acknowledge{Acknowledgements:} The author would like to thank Anthony Quas for the prodding to consider how elegant and satisfying this collection of ideas really is.
\endacknowledge

\bibliographystyle{abbrv}
\bibliography{expos-refs}

\begin{biog}
\item[Joseph Horan] is a Ph.D.\ candidate at the University of Victoria, for now. He is interested in dynamical systems, ergodic theory, and mathematics education, and is very excited to see beautiful mathematics.
\begin{affil}
Department of Mathematics and Statistics, University of Victoria, Victoria, BC, Canada V8P 5C2\\
jahoran@uvic.ca
\end{affil}
\end{biog}

\end{document}